\documentclass[11pt]{article}

\title{The global resilience of Hamiltonicity in $G(n,p)$}

\author{
Yahav Alon
\thanks{School of Mathematical Sciences, Raymond and Beverly Sackler Faculty of Exact Sciences, Tel Aviv University,
Tel Aviv, 6997801, Israel. Email: yahavalo@tauex.tau.ac.il.}
}

\usepackage{tikz}
\usepackage{mathtools}
\usepackage{verbatim}
\usetikzlibrary{arrows,%
                petri,%
                topaths}%
\usepackage[position=top]{subfig}
\usepackage{amsmath,amsthm, amssymb,latexsym, amsfonts}
\usepackage{algorithm}
\usepackage{enumitem}

\begin{document}
\maketitle
\newtheorem{thm}{Theorem}
\newtheorem{propos}{Proposition}
\newtheorem{defin}{Definition}
\newtheorem{lemma}{Lemma}[section]
\newtheorem{corol}{Corollary}[section]
\newtheorem{corol*}{Corollary}
\newtheorem{thmtool}{Theorem}[section]
\newtheorem{corollary}[thmtool]{Corollary}
\newtheorem{lem}[thmtool]{Lemma}
\newtheorem{defi}[thmtool]{Definition}
\newtheorem{prop}[thmtool]{Proposition}
\newtheorem{clm}[thmtool]{Claim}
\newtheorem{conjecture}{Conjecture}
\newtheorem{problem}{Problem}
\newcommand{\Proof}{\noindent{\bf Proof.}\ \ }
\newcommand{\Remarks}{\noindent{\bf Remarks:}\ \ }
\newcommand{\Remark}{\noindent{\bf Remark:}\ \ }

\newcommand{\Dist}[1]{\mathsf{#1}}
\newcommand{\Bin}{\Dist{Bin}}
\newcommand{\HH}{\mathcal{H}}
\newcommand{\PP}{\mathcal{P}}
\newcommand{\pr}{\mathbb{P}}
\newcommand{\sm}{\text{SMALL}}

\begin{abstract}
Denote by $r_g(G,\HH)$ the \emph{global resilience} of a graph $G$ with respect to Hamiltonicity. That is, $r_g(G,\HH)$ is the minimal $r$ for which there exists a subgraph $H\subseteq G$ with $r$ edges, such that $G\setminus H$ is not Hamiltonian. We show that if $p$ is above the Hamiltonicity threshold and $G\sim G(n,p)$ then, with high probability\footnote{We say that a sequence of events $(A_n)_{n=1}^{\infty}$ occurs with high probability if $\lim _{n\to \infty}\pr (A_n) =1$.}, $r_g(G,\HH)=\delta (G)-1$. This is easily extended to the full interval: for every $p(n)\in [0,1]$, if $G\sim G(n,p)$ then, with high probability, $r_g(G,\HH)= \max \{ 0,\delta (G)-1 \}$.
\end{abstract}

\section{Introduction} \label{sec-intro} 

Let $\PP$ be a monotone increasing graph property. For a graph $G$, we define the \emph{global resilience} of $G$ with respect to $\PP$, denoted $r_g(G,\PP)$, as follows.
$$
r_g(G,\PP) \coloneqq \min \left\{ m\in \mathbb{N} \mid \exists H\subseteq G: e(H)=m, G \setminus H \text{ is not in } \PP \right\} .
$$
That is, $r_g(G,\PP)$ is the minimal number of edge removals from $G$ such that the resulting graph does not satisfy $\PP$. This notion serves as a measure of how ``strongly" $G$ satisfies the property $\PP$, by its distance from the closest graph outside of $\PP$. It is by no means a new notion, and many long standing results in extremal graph theory can be expressed by it. For example, the extremal number $\text{ex}(n,G)$ can be expressed as $\binom{n}{2}-r_g(K_n,\PP _G)$, where $\PP _G$ denotes the property of containing a copy of $G$ as a subgraph. Tur{\'a}n's theorem can now be stated as: for every $n\ge r\ge 3$ integers, $r_g(K_n,\PP _{K_r}) = (1+o(1))\cdot \frac{n^2}{2(r-1)}$.

We denote by $\HH$ the property of Hamiltonicity. It is trivial to see that, for every graph $G$ with $\delta (G) \ge 1$, one has $r_g(G,\HH) \le \delta (G) -1$. Indeed, one can ensure that $G\setminus H$ contains no Hamilton cycle by having $H$ contain all the edges incident in $G$ to some vertex $v\in V(G)$ but one. By choosing $v$ to be a vertex with minimum degree, this trivial bound is achieved.

In this paper we show that this trivial upper bound is, in fact, the typical exact value of $r_g(G,\HH)$ when $G$ is a random graph, by proving the following theorem.\footnote{Here and later the logarithms have natural base.}

\begin{thm}\label{main}
Let $G\sim G(n,p)$, where $p=p(n)$ satisfies $np - \log n - \log \log n \to \infty$. Then with high probability $r_g(G,\HH) = \delta (G) -1$.
\end{thm}

 Ore \cite{ORE} proved that every $n$-vertex graph with at least $\binom{n-1}{2}+2$ edges is Hamiltonian. Restated in terms of the global resilience, this implies that $r_g(K_n,\HH ) = n-2$, and thus the theorem holds for $p=1$.
 
To cover the complete range of $p$, the following corollary is easily derived from Theorem \ref{main}.

\begin{corol*} \label{main-corol}
Let $p(n)\in [0,1]$ and $G\sim G(n,p)$. Then with high probability $r_g(G,\HH) = \max \{0,\delta (G) -1\}$.
\end{corol*}

\begin{proof}
Consider separately three ranges of $p$:

\noindent \textbf{Sub-critical:} if $np-\log n - \log \log n \to -\infty$, then with high probability $\delta (G) = 0$ and $G$ is not Hamiltonian, and therefore $r_g(G,\HH) = 0 = \max \{0,\delta (G) -1\}$.

\noindent \textbf{Critical:} if $np-\log n - \log \log n \to c$, then with high probability either $\delta (G) \le 1$, or $\delta (G) = 2$ and $G$ is Hamiltonian (a consequence of Ajtai, Koml{\'o}s, Szemer{\'e}di \cite{AKS85}, Bollob{\'a}s \cite{B84}). In the first case indeed $r_g(G,\HH) = 0 = \max \{0,\delta (G) -1\}$. In the second case, because $G$ is Hamiltonian, $1 \le r_g(G,\HH) \le \delta (G)-1$, as this is true for every Hamiltonian graph. But in this case $\max \{0,\delta (G) -1\} = 1$ and therefore indeed $ r_g(G,\HH)  = \max \{0,\delta (G) -1\}$.

\noindent \textbf{Super-critical:} the case $np-\log n - \log \log n \to \infty$ is covered by Theorem \ref{main}, since in the super-critical regime, with high probability, $\max \{0,\delta (G) -1\} = \delta (G)-1$.

\end{proof}

\subsection*{Related work}

The \emph{local resilience} of a property is a similar notion of resilience, that, with respect to Hamiltonicity in random graphs, has been more thoroughly studied. Denote by $r_{\ell}(G,\PP )$ the local resilience of $G$ with respect to $\PP$, defined as the minimal value of $m$ such that there is a graph $H$ with $\Delta (H) \le m$, and $G\setminus H$ does not satisfy $\PP$. Sudakov and Vu \cite{SV} showed that there is $C>0$ such that, for every $\varepsilon, \delta >0$, if $p \ge \frac{C\log ^{2+\delta}n}{n}$ then with high probability $r_{\ell}(G(n,p),\HH) \ge (1/2-\varepsilon )np$. Closer to the Hamiltonicity threshold, Ben-Shimon, Krivelevich and Sudakov \cite{BKS} showed that if $p\le \frac{1.02\log n}{n}$ is above the Hamiltonicity threshold then, with high probability, $r_{\ell}(G(n,p),\HH ) = \delta (G) -1$. Since it is always true that $r_{\ell}(G,\PP ) \le r_g(G,\PP )$, an immediate consequence of this is that Theorem \ref{main} holds when $p\le \frac{1.02\log n}{n}$.

One can also consider measures of resilience where the limitations on the degrees of the subtracted subgraph $H$ depend on the degrees of $G$. Lee and Sudakov \cite{LS} showed that for every $\varepsilon >0$, if $C>0$ is large enough with respect to $\varepsilon$, $p\ge \frac{C\log n}{n}$ and $G\sim G(n,p)$ then, with high probability, $G\setminus H$ is Hamiltonian for every $H\subseteq G$ such that $\delta (G\setminus H) \ge (1/2+\varepsilon )np$. Another notion of local resilience one can consider is $\alpha$-resilience. We say that $G$ is $\alpha$-resilient with respect to $\PP$ if $G\setminus H$ has $\PP$ for every $H\subseteq G$ such that $d_H(v)\le \alpha\cdot d_G(v)$ for every $v\in V(G)$. Montgomery \cite{MONT}, and independently Nenadov, Steger and Truji{\'c} \cite{NST} showed that in the \emph{random graph process} model $\{G _m \}_{m\ge 0}$, if $m$ is past the hitting time of Hamiltonicity then, with high probability, $G_m$ is $(1/2-\varepsilon )$-resilient. Nenadov et al. additionally extended this result below the hitting time, and showed that the 2-core of $G_m$ is also $(1/2-\varepsilon )$-resilient with respect to Hamiltonicity with high probability, given that $m\ge (1/6+\varepsilon )n\log n$.

For further reading on various measures of resilience of graph properties we refer to Sudakov's survey on the subject \cite{SUDA}.

Alon and Krivelevich \cite{AK} proved that for $G\sim G(n,p)$, $p$ above the Hamiltonicity threshold, one has
$\pr (G \notin \HH) = (1+o(1))\cdot \pr (\delta (G) <2 ).$ Informally, this result suggests that the greatest obstacle  (probability-wise) for a random graph to be Hamiltonian is the minimum degree. In a sense, Theorem \ref{main} shows something similar, by showing that, with high probability, the nearest non-Hamiltonian graph to $G$ indeed has minimum degree less than 2.

\section{Preliminaries} \label{sec-per}

The following graph theoretic notation is used.
For a graph $G=(V,E)$ and two disjoint vertex subsets $U,W\subseteq V$, we let $E_G(U,W)$ denote the set of edges of $G$ adjacent to exactly one vertex from $U$ and one vertex from $V$, and let $e_G(U,W)=|E_G(U,W)|$.
Similarly, $E_G(U)$ denotes the set of edges spanned by a subset $U$ of $V$, and $e_G(U)$ stands for $|E_G(U)|$, and $E_G(v)$ denotes $E_G(\{v\},V\setminus \{v\})$.
The (external) neighbourhood of a vertex subset $U$, denoted by $N_G(U)$, is the set of vertices in $V\setminus U$ adjacent to a vertex of $U$, and for a vertex $v\in V$ we set $N_G(v)=N_G(\{v\})$.
The degree of a vertex $v\in V$, denoted by $d_G(v)$, is its number of incident edges.

While using the above notation we occasionally omit $G$ if the identity of the graph $G$ is clear from the context.

We suppress the rounding notation occasionally to simplify the presentation.

\subsection*{Auxiliary results}

\begin{defi}\label{expander}
Let $\alpha>0$ and $k$ a positive integer. A graph $G$ is a {\em $(k,\alpha )$-expander} if $|N_G(U)|\ge \alpha |U|$ for every vertex subset $U\subset V(G)$, $|U|\le k$.
\end{defi}

\begin{defi}\label{def-booster}
Let $G$ be a graph. A non-edge $\{u,v\}\in E(G)$ is called a {\em booster} if the graph $G^{\prime}$ with edge set $E(G^{\prime})=E(G)\cup \lbrace \{u,v\} \rbrace$ is either Hamiltonian or has a path longer than a longest path of $G$.
\end{defi}

\begin{lemma}\label{lemma-posa}
\emph{(P{\'o}sa 1976 \cite{POS})} Let $G$ be a connected non--Hamiltonian graph, and assume that $G$ is a $(k,2)$--expander. Then $G$ has at least $\frac{(k+1)^2}{2}$ boosters.
\end{lemma}

\begin{defi}\label{def-hamconn}
A graph $G$ is \emph{Hamilton--connected} if for every two vertices $u,v\in V(G)$, $G$ contains a Hamilton path with $u,v$ as its two endpoints.
\end{defi}

\begin{thm}\label{chvatal-erdos}
{\em (Chv{\'a}tal--Erd\H{o}s Theorem \cite{CE})} Let $G=(V,E)$ be a graph such that $\alpha (G) < \kappa (G)$. Then $G$ is Hamilton--connected.
\end{thm}

\subsection*{Useful inequalities}

\begin{lemma}\label{coefficients}
Let $1\leq l \leq k \le n$ be integers. Then the following inequalities hold:
\begin{enumerate}
\item $\binom{n}{k} \le \left( \frac{en}{k} \right) ^k$\,;
\item $\frac{\binom{n-l}{k}}{\binom{n}{k}} \le e^{-\frac{l\cdot k}{n}}$\,.
\end{enumerate}
\end{lemma}

\begin{lemma}\label{binom-rv}
Let $1\le k \le n$ be integers, $0 < p < 1$, and let $X\sim Bin(n,p)$. Then the following inequalities hold:

\begin{enumerate}
\item $\pr (X \ge k) \le \left( \frac{enp}{k} \right) ^k$\,;
\item $\pr (X = k) \le \left( \frac{enp}{k(1-p)} \right) ^k\cdot e^{-np}$\,.
\end{enumerate}

\noindent If, additionally, $k\le np$, then 

\begin{enumerate}[resume]
\item $\pr (X \le k) \le (k+1)\cdot \left( \frac{enp}{k(1-p)} \right) ^k\cdot e^{-np}$\,.
\end{enumerate}

\end{lemma}

\section{Proof of Theorem \ref{main}}

In this section we present a proof of Theorem \ref{main}. We prove it separately for two different ranges of the probability $p$, as the typical properties of the random graph in these two regimes is fairly different. In Section \ref{sparse} we prove the claim in the sparse regime $p\le n^{-0.4}$, and in Section \ref{dense} we prove it the the dense regime $n^{-0.4} \le p \le 1$.

\subsection{Sparse case} \label{sparse}

\begin{thm} \label{thm-sprs-case}
Let $G\sim G(n,p)$ where $np-\log n - \log \log n \to \infty$ and $p\le n^{-0.4}$. Then with high probability $r_g(G,\HH) = \delta (G) -1$.
\end{thm}

\begin{proof}

Let $d_0 = 0.001 np$ and $\sm (G) = \{ v\in V(G) \mid d_G(v) < d_0 \}$.

\begin{lemma} \label{lemma:props}
With high probability $G$ has the following properties.

\begin{enumerate}[{label=\textbf{(P\arabic*)}}]
    \item
      \label{item:minmax-degrees}
      $\delta (G) \ge 2$ and $\Delta (G) \le 5np$;
    \item
      \label{item:SMALL-is-small}
      $|\sm(G)| \le n^{0.1}$;
    \item
      \label{item:SMALL-are-far}
      $G$ does not contain a path of length at most 4 with both its endpoints in $\sm(G)$;
    \item
      \label{item:not-too-many-edges}
      every $U\subseteq V(G)$ with $\frac{1}{2}d_0 \le |U| \le \frac{5n}{\sqrt{np}}$ spans at most $\frac{1}{15}d_0|U|$ edges;
    \item
      \label{item:not-few-many-edges}
      if $U,W\subseteq V(G)$ are disjoint and $|U|= |W| = \frac{n}{\sqrt{np}}$ then $e(U,W) \ge n/2$.
\end{enumerate}

\end{lemma}

\begin{proof}

For each of the given properties, we bound the probability that $G\sim G(n,p)$ fails to uphold it.

\begin{itemize}

	\item[\ref{item:minmax-degrees}.]
		Since $p$ is above the Hamiltonicity threshold, which is equal to the threshold of the property $\delta (G) \ge 2$, the first part is obvious. For the second part, by the union bound we get
		$$
		\pr \left( \Delta (G) \ge 5np \right) \le n\cdot \pr \left( \Bin (n-1,p) \ge 5np \right) \le n\cdot \left( \frac{e(n-1)p}{5np} \right) ^{5np} \le  n^{-2}.
		$$

    \item[\ref{item:SMALL-is-small}.]
      The probability that $|\sm(G)| \ge n^{0.1}$ is at most the probabiliy that there is a set of size $s \coloneqq n^{0.1}$ with less than $d_0\cdot s$ outgoing edges. Therefore
      
	\begin{eqnarray*}
	\pr \left( |\sm (G)| \ge n^{0.1} \right) & \le & \binom{n}{s}\cdot \pr \left( \Bin \left( s(n-s), p \right) < d_0 \cdot s \right) \\
	& \le & \binom{n}{s}\cdot d_0s\cdot \pr \left( \Bin \left( s(n-s), p \right) = d_0 \cdot s \right) \\
	& \le & \left( \frac{en}{s} \right) ^s \cdot d_0s \cdot \left( \frac{es(n-s)p}{d_0s(1-p)} \right) ^{d_0s}\cdot e^{-s(n-s)p} \\
	& \le & n^{0.9s} \cdot d_0s \cdot 3000 ^{d_0s}\cdot e^{-0.95\cdot snp} \\
	& \le & \exp \left( s\cdot \log n \cdot \left( o(1) + 0.9 + 0.001\cdot \log 3000 - 0.95 \right) \right) \\
	& = & o(1).
	\end{eqnarray*}
      
    \item[\ref{item:SMALL-are-far}.]
      Given $u,v\in V(G)$ and a path $P$ of length $\ell$ between them, the probability that $u,v\in \sm (G)$ and $P\subseteq G$ is at most the probability that $P\subseteq G$ and $\{ u,v \}$ has less than $2d_0$ outgoing edges that are not part of $P$, which is at most
      
      $$
      p^{\ell}\cdot 2d_0 \cdot \left( \frac{2enp}{2d_0(1-p)} \right) ^{2d_0}\cdot e^{-2(n-3)p} \le  p^{\ell}\cdot e^{-1.9np}.
      $$
      By the union bound, the probability that there is a path $P\subseteq G$ of length at most 4 with both endpoints in $\sm (G)$ is at most
      $$
      \sum _{\ell = 1}^4 \binom{n}{\ell +1} p^{\ell}\cdot e^{-1.9np} = o(1).
      $$
      
    \item[\ref{item:not-too-many-edges}.]
      The probability that there is a set $U\subseteq V(G)$ of size $k\ge \frac{1}{2}d_0$ that contradicts \ref{item:not-too-many-edges} is at most
      
      $$
       \binom{n}{k}\cdot \pr \left( \Bin \left( \binom{k}{2},p \right) \ge \frac{1}{15}d_0k \right) \le 
       \left( \frac{en}{k} \right) ^k \cdot \left( \frac{15ek^2p}{2d_0k} \right) ^{0.1d_0k} \le (np)^{-0.04d_0k},
      $$
      
      where the last inequality is due to the fact that $k\le \frac{5n}{\sqrt{np}}$. Therefore, the probability that $G$ does not have \ref{item:not-too-many-edges} is at most
      $$
      \sum _{k=d_0/2}^{5n/\sqrt{np}} (np)^{-0.04d_0k} = (1+o(1))(np)^{-0.02{d_0}^2} = o(1).
      $$
      
    \item[\ref{item:not-few-many-edges}.]
      By the union bound, the probability that there are $U,W\subseteq V(G)$ of size $\frac{n}{\sqrt{np}}$ with less than $\frac{1}{2}n$ edges between them is at most
      
      $$
      \binom{n}{\frac{n}{\sqrt{np}}}^2\cdot \frac{1}{2}n \cdot \pr \left( \Bin \left( \frac{n}{p},p \right) = \frac{1}{2}n \right) \le n\cdot \left( enp \right) ^{\frac{2n}{\sqrt{np}}} \cdot \left( \frac{2en}{(1-p)n} \right)^{\frac{1}{2}n}\cdot e^{-n} = o(1).
      $$
      
\end{itemize}

\end{proof}

\begin{lemma} \label{lemma:expander}
With high probability, for every subgraph $H\subseteq G$ with $e(H) = \delta (G) -2$, the graph $G\setminus H$ contains a subgraph $\Gamma _0$ that is an $\left( \frac{n}{4},2 \right)$-expander with at most $d_0n$ edges.
\end{lemma}

\begin{proof}

We prove this by showing that if $G$ satisfies properties \ref{item:minmax-degrees}-\ref{item:not-few-many-edges} then, for every $H\subseteq G$ with $\delta (G) -2$ edges, $G\setminus H$ contains a subgraph $\Gamma _0$ as required. To this end we consider a random subgraph of $G\setminus H$ with at most $d_0n$ edges and show that it is an $\left( \frac{n}{4},2 \right)$-expander with positive probability, which implies existence.

Construct a random subgraph of $G\setminus H$ as follows. For every $v\in V(G)$ set $E_v$ to be $E_{G\setminus H}(v)$ in the case $d_{G\setminus H}(v) \le d_0$, and otherwise set $E_v$ to be a subset of $E_{G\setminus H}(v)$ of size $d_0$, chosen uniformly at random and independently of all other choices. The random subgraph $\Gamma$ is the $G\setminus H$-subgraph with edge set $\bigcup _{v\in V(G)}E_v$. Observe that the minimum degree of a graph $\Gamma$ drawn this way is at least $\min \{\delta (G\setminus H),d_0 \} \ge 2$, that $d_{\Gamma}(v)=d_{G\setminus H}(v)$ for every $v\in \sm (G)$, and that $e(\Gamma ) \le d_0n$.

We bound from above the probability that $\Gamma$ contains a subset $U$ with at most $n/4$ vertices with less than $2|U|$ neighbours. Let $|U|=k\le \frac{n}{4}$ and denote $U_1 = U\cap \sm (G),U_2 = U\setminus U_1$ and $k_1,k_2$ the sizes of $U_1,U_2$ respectively. Observe that \ref{item:SMALL-are-far} implies that \emph{(i)} every vertex in $U_2$ has at most one neighbour in $U_1 \cup N_G(U_1)$, and therefore $|N_{\Gamma}(U_2) \cap (U_1 \cup N_{\Gamma}(U_1)| \le k_2$; and \emph{(ii)} distinct vertices in $\sm (G)$ have non-intersecting neighbourhoods, and therefore $|N_{\Gamma}(U_1)| \ge \delta(\Gamma)\cdot k_1 \ge 2k_1$.

First we show that if $k_2 \le \frac{n}{\sqrt{np}}$ then $|N_{\Gamma}(U)|\ge 2|U|$ with probability 1. We separate the proof into different cases according to the value of $k_2$.

\begin{enumerate}

\item
$k_2=1$.
If $k_1=0$ then $U$ is a singleton, and $N_{\Gamma}(U)$ contains at least two vertices since $\delta (\Gamma) \ge 2$.

Otherwise, $k_1>0$ and in particular $\sm (G)$ is not empty, so $\delta (G) <d_0$ and
\begin{eqnarray*}
|N_{\Gamma}(U)| & \ge & |N_{\Gamma}(U_1)\setminus U_2| + |N_{\Gamma}(U_2)\setminus (N_{\Gamma}(U_1) \cup U_1) | \\
& \ge & \delta (\Gamma )\cdot k_1 -1 + d_0-(\delta (\Gamma ) -2) - 1 \\
& \ge & 2k_1+2=2|U|.
\end{eqnarray*}

\item
$2\le k_2 \le \frac{1}{10}d_0$.
Since there are at least two vertices, there is a vertex $v\in U_2$ such that
$$
d_{G\setminus H}(v) \ge d_G(v)-\frac{1}{2}(\delta (G)-2) -1 \ge \frac{1}{2}d_G(v) \ge \frac{1}{2}d_0,
$$ 
and therefore also $d_{\Gamma}(v) \ge \frac{1}{2}d_0$, and $e_{\Gamma}(v,V(G)\setminus U_2) \ge \frac{1}{2}d_0-k_2 \ge \frac{2}{5}d_0$. We get
\begin{eqnarray*}
|N_{\Gamma}(U)| & \ge & |N_{\Gamma}(U_1)\setminus U_2| + |N_{\Gamma}(U_2)\setminus (N_{\Gamma}(U_1) \cup U_1) | \\
& \ge & 2 k_1 - k_2 + \frac{2}{5}d_0- k_2 \\
& \ge & 2k_1+2k_2=2|U|.
\end{eqnarray*}

\item
$\frac{1}{10}d_0 \le k_2 \le \frac{n}{\sqrt{np}}$. In this case $|N_{\Gamma}(U_2)|\ge 4k_2$. Indeed, if $|N_{\Gamma}(U_2)|\le 4k_2$ then $U_2 \cup N_{\Gamma}(U_2)$ is contained in a set of size $5k_2$, which is between $\frac{1
}{2}d_0$ and $\frac{5n}{\sqrt{np}}$, that spans at least $\frac{1}{2}d_0k_2-e(H)\ge \frac{1}{15}d_0\cdot \left( 5k_2 \right)$ edges in $G$, a contradiction to \ref{item:not-too-many-edges}. We get 
\begin{eqnarray*}
|N_{\Gamma}(U)| & \ge & |N_{\Gamma}(U_1)\setminus U_2| + |N_{\Gamma}(U_2)\setminus (N_{\Gamma}(U_1) \cup U_1) | \\
& \ge & 2 k_1 - k_2 + 4k_2- k_2 \\
& \ge & 2k_1+2k_2=2|U|.
\end{eqnarray*}

\end{enumerate}

For the remaining case $\frac{n}{\sqrt{np}}\le k_2 \le \frac{n}{4}$ we show that $|N_{\Gamma}(U)| \ge 2|U|$ with positive probability. Indeed, assume that $|N_{\Gamma}(U)| < 2|U|$, then $|V(G)\setminus (U \cap N_{\Gamma}(U))| \ge \frac{1}{5}n$. In particular, there are disjoint sets $U'\subseteq U$ and $W\subseteq V(G)\setminus (U \cap N_{\Gamma}(U))$, each of size $\frac{n}{\sqrt{np}}$, such that $e_{\Gamma}(U',W)=0$. Observe that by \ref{item:not-few-many-edges}, $e_{G\setminus H}(U',W) \ge \frac{1}{2}n-\delta(G) \ge \frac{1}{3}n$. For a given pair of subsets $U',W$, the probability for this is at most

\begin{eqnarray*}
\prod_{u\in U'}\pr (e_{\Gamma}(u,W)=0) & \le & \prod _{u \in U'} \frac{\binom{d_{G\setminus H}(u) - e_{G\setminus H}(u,W)}{d_0}}{\binom{d_{G\setminus H}(u)}{d_0}} \\
& \le & \prod _{u \in U'} e^{-\frac{d_0\cdot e_{G\setminus H}(u,W)}{d_{G\setminus H}(u)}}\\
& \le & \exp \left( {-\frac{d_0}{\Delta (G)}\cdot e_{G\setminus H}(U,W)} \right) \\
& \le & \exp \left( -\frac{1}{15000}n \right) ,
\end{eqnarray*}
Where in the last inequality we used the fact that $G$ has \ref{item:minmax-degrees}, and therefore $\Delta (G) \le 5np$.

Since there are $\exp (o(n))$ pairs of subsets $U',W$ of size $\frac{n}{\sqrt{np}}$, by the union bound the probability that two subsets of this size with no edges between them in $\Gamma$ exist is of order $o(1)$. Consequently, the random subgraph $\Gamma$ is an $\left( \frac{n}{4},2 \right)$-expander with probability $1-o(1)$, implying that $G\setminus H$ contains a sparse expander, as claimed.

\end{proof}

\begin{lemma}\label{lemma:boosters}
With high probability, for every subgraph $H\subseteq G$ with $e(H) = \delta (G) -2$ and every non-Hamiltonian $\left( \frac{n}{4},2 \right)$-expander $\Gamma \subseteq G$ with $e(\Gamma )\le 2d_0n$, the graph $G\setminus (H\cup \Gamma )$ contains a booster with respect to $\Gamma$.
\end{lemma}

\begin{proof}
By Lemma \ref{lemma-posa}, a non-Hamiltonian $\left( \frac{n}{4},2 \right)$-expander has at least $\frac{n^2}{32}$ boosters. For a given subgraph $H$ with $\delta (G)-2$ edges and a given expander $\Gamma$, the probability that none of the many boosters are in $G\setminus H$ is at most

$$
\pr \left( \Bin \left( \frac{n^2}{32},p \right) \le e(H) \right) \le \delta (G) \cdot \left( \frac{en^2p}{32\cdot\delta (G)\cdot (1-p)} \right) ^{\delta (G)} \cdot e^{-\frac{n^2p}{32}} \le e^{-\frac{n^2p}{33}}.
$$

By the union bound, the probability that there is an expander subgraph $\Gamma \subseteq G$ with at most $2d_0n$ edges, and no boosters with respect to $\Gamma$ in $G$, is at most

$$
\sum _{k=1}^{2d_0n}\binom{\binom{n}{2}}{k}\cdot p^k \cdot e^{-\frac{n^2p}{33}}
\le 2d_0n\cdot \left( \frac{enp}{4d_0} \right) ^{2d_0n} \cdot e^{-\frac{n^2p}{33}} = \exp \left( n^2p \cdot \left( o(1) + \frac{\log (250e)}{500} - \frac{1}{33} \right) \right) =o(1).
$$

\end{proof}

\begin{corol}
With high probability, for every subgraph $H\subseteq G$ with $e(H) = \delta (G) -2$ the graph $G\setminus H$ is Hamiltonian.
\end{corol}

Indeed, assume that $G$ satisfies the properties in the assertions of Lemma \ref{lemma:expander} and Lemma \ref{lemma:boosters}, an event which occurs with high probability. Then, given $H\subseteq G$ with $e(H) = \delta (G) -2$, the subgraph $G\setminus H$ contains an $\left( \frac{n}{4},2 \right)$-expander subgraph $\Gamma_0$ with at most $d_0n$ edges. Then, while $\Gamma _i$ is not Hamiltonian, $G\setminus H$ contains a booster with respect to it, which we add to $\Gamma _i$ to obtain $\Gamma _{i+1}$. Repeating this for at most $n$ steps we obtain a Hamiltonian subgraph of $G\setminus H$.
\end{proof}

\subsection{Dense case}  \label{dense}

\begin{thm}
Let $G\sim G(n,p)$ where $n^{-0.4} \le p\le 1$. Then with high probability $r_g(G,\HH) = \delta (G) -1$.
\end{thm}

\begin{proof}

\begin{lemma}  \label{denselemma}
With high probability $G$ has the following properties.

\begin{enumerate}[{label=\textbf{(Q\arabic*)}}]
    \item
      \label{item:dense-min-degree}
      $\delta (G) \ge \frac{1}{2}np$;
    \item
      \label{item:dense-independent-set}
      if $U\subseteq V(G)$ and $|U|=\frac{1}{8}np$ then $e(U) \ge n$;
    \item
      \label{item:dense-vertex-conn}
      if $U,W\subseteq V(G)$ are disjoint and $|U|= |W| = \frac{1}{8}np$ then $e(U,W) \ge n$.
\end{enumerate}

\end{lemma}

\begin{proof}
An upper bound of order $o(1)$ on the probability that $G\sim G(n,p)$ fails to uphold any of the three properties follows from applying the union bound and standard bounds on binomial distributions.

\end{proof}

The proof of Theorem \ref{dense} now follows from Lemma \ref{denselemma}. We prove that if $G$ has properties \ref{item:dense-min-degree}-\ref{item:dense-vertex-conn} then $G\setminus H$ is Hamiltonian for every such $H$ with $e(H) = \delta (G)-2$.

Let $v_0\in V(G)$ be a vertex with $d_{G\setminus H}(v_0)=\delta (G\setminus H)$ and denote $G' \coloneqq (G\setminus H)-v_0$. Then $\delta (G') \ge \frac{1}{2}\delta (G)$. We now claim that $\kappa (G') > \alpha (G')$, and therefore by Theorem \ref{chvatal-erdos} we conclude that $G'$ is Hamilton-connected. Since $d_{G\setminus H}(v_0) \ge 2$ this implies that $G\setminus H$ is Hamiltonian.

Indeed, by \ref{item:dense-independent-set}, every vertex subset with $\frac{1}{8}np$ vertices spans at least $n-\delta(G)>0$ edges in $G'$, and therefore $\alpha (G') < \frac{1}{8}np$.

On the other hand, let $V(G') = U\cup X \cup W$ be a partition of $V(G')$ such that $U,W$ are non-empty and $e_{G'}(U,W) =0$. Assume without loss of generality that $|U|\le |W|$, Then $|U| < \frac{1}{8}np$, since otherwise by \ref{item:dense-vertex-conn} we have $e_{G'}(U,W) \ge n-\delta (G) > 0$. Additionally, by \ref{item:dense-min-degree} we have $\delta (G') \ge \frac{1}{2}\delta (G) \ge \frac{1}{4}np$. Since $U\cup X$ contains all of the neighbours of a vertex $u\in U$ we get $|X| \ge d_{G'}(u) - |U| \ge \frac{1}{4}np-\frac{1}{8}np =\frac{1}{8}np$. Therefore $\kappa (G') \ge \frac{1}{8}np$.

\end{proof}

\section{Concluding remarks}

We note that the proof of Theorem \ref{main} (and, in fact, Corollary \ref{main-corol}) presented in this paper can be adjusted slightly to prove the following statement, where here PM denotes the property of containing a perfect matching.

\begin{prop}
Let $p(n)\in [0,1]$ and $G\sim G(n,p)$, where $n$ is even. Then with high probability $r_g(G,\text{PM}) = \delta (G) $.
\end{prop}

For the critical and sub-critical regimes the same reasoning as in the proof of Corollary \ref{main-corol} can be applied, where for the critical regime we replace the probability that a graph is Hamiltonian with the result by Erd\H{o}s and R\'{e}nyi \cite{ER} regarding perfect matchings. Also observe that, given Theorem \ref{main}, only the sparse case of the super-critical regime requires a proof. Indeed, in the dense case, if $G\sim G(n,p)$ then, with high probability, for every $H$ with $e(H) \le \delta (G)-1$ the graph $G \setminus H$ contains a Hamilton \textbf{path}, also implying that it contains a perfect matching. 

The sparse case of the super-critical regime can be proved by applying some small adjustments to the proof of Theorem \ref{thm-sprs-case}. Here, property \ref{item:minmax-degrees} in Lemma \ref{lemma:props} should be changed to state that $\delta (G) \ge 1$. A slight adjustment to Lemma \ref{lemma:expander} then shows that, with high probability, $G\setminus H$ contains a sparse $(n/4,1)$-expander for any $H$ with at most $\delta (G) -1$ edges. The last part of the proof is identical, as $(k,1)$-expanders are also known to have $\frac{(k+1)^2}{2}$ boosters (with respect to maximum size matchings, rather than maximum length paths. See e.g. \cite{FK}, Lemma 6.3).

\subsection*{Acknowledgement}
The author would like to thank Professor Michael Krivelevich for his support and valuable advice during the writing of this paper.

\end{document}